\documentclass[12pt]{article} 

\usepackage{amsmath} 
\usepackage{amsthm}
\usepackage{amssymb}
\usepackage{enumitem} 
\usepackage{mathtools}
\usepackage[hidelinks]{hyperref}
\usepackage[utf8]{inputenc}
\usepackage[pagewise]{lineno}
\usepackage[english]{babel}
\usepackage{authblk}
\usepackage{enumitem}
\usepackage{titlesec}
\usepackage{graphicx}
\usepackage{subcaption}
\usepackage{hyperref}
\usepackage{mwe}
\usepackage{bm}
\usepackage{verbatim}
\usepackage[title]{appendix}
\usepackage{tikz-cd}
\usepackage{caption}

\usepackage{array}  % For additional table functionalities
\usepackage{multirow}  % For multi-row cells
\usepackage{geometry}  % For adjusting margins
\usepackage{diagbox}  % For diagonal lines in table cells

\usepackage{geometry}
\makeatletter
\BeforeBeginEnvironment{proof}{%
  \def\@Opargbegintheorem#1#2#3#4{#4\trivlist
      \item[]{#3#2\@thmcounterend\ }}%
}
\AfterEndEnvironment{proof}{%
  \def\@Opargbegintheorem#1#2#3#4{#4\trivlist
      \item[\hskip\labelsep{#3#1}]{#3(#2)\@thmcounterend\ }}%
}
 \newtheorem{thm}{Theorem}[subsection]%{\bfseries}{\itshape}
 \newtheorem{cor}[thm]{Corollary}%{\bfseries}{\itshape}
 \newtheorem{lem}[thm]{Lemma}%{\bfseries}{\itshape}
 \newtheorem{exmp}{Example}[subsection]%{\bfseries}{\itshape}
\hypersetup{
    colorlinks=true, %set true if you want colored links
    linktoc=all,     %set to all if you want both sections and subsections linked
    linkcolor = blue,%choose some color if you want links to stand out
}

\providecommand{\keywords}[1]
{
	\small	
	\textbf{\text{Keywords:}} #1
}

\providecommand{\subclass}[1]
{
	\small	
	\textbf{\text{MSC(2020):}} #1
}

\begin{document}

\title{Some Generalizations of Totient Function with Elementary Symmetric Sums}
    \author[1]{Udvas Acharjee\thanks{udvasacharjee@sssihl.edu.in}}
\author[1]{N. Uday Kiran\thanks{Corresponding author: nudaykiran@sssihl.edu.in}}
\affil[1]{Department of Mathematics and Computer Science, Sri Sathya Sai Institute of Higher Learning, Puttaparthi, India.\thanks{Dedicated to Bhagawan Sri Sathya Sai Baba.}}

	\date{\today}
    \maketitle
\begin{abstract}

In a recent work, T\'oth introduced a $k$-dimensional generalized Euler totient function using the first and $k$-th elementary symmetric sums $e_1(\mathbf{x}) = x_1 + \cdots + x_k$ and $e_k(\mathbf{x}) = x_1 \cdots x_k$ respectively. We extend this framework by incorporating the second elementary symmetric sum $e_2(\mathbf{x}) = x_1x_2+x_1x_3+\dots+x_{k-1}x_k$, deriving new generalized totient functions with explicit product forms and Menon-type identities. We demonstrate that these totient functions enable counting solutions to restricted linear congruences $a_1 x_1 + \cdots + a_k x_k \equiv b \pmod{n}$ under $\gcd$-constraints on quadratic forms like $e_2(\mathbf{x})$, when $\gcd(b,n)=1$.
 

   \subclass{11A05 $\cdot$ 11A07 $\cdot$ 11A25}
\end{abstract}
 
\section{Introduction}
Recently, T\'oth~\cite{toth2022another} studied a generalization of the Euler totient function corresponding to the set $\{e_1, e_k\}$, where $e_1(\mathbf{x}) = x_1+\dots+x_k$ and $e_k(\mathbf{x}) = x_1\dots x_k$ are the first and $k$th elementary symmetric sums and $\mathbf{x}$ denotes $(x_1,x_2,\dots, x_k)$. This generalizes earlier results of Arai~\cite{arai1961e1460} and Carlitz~\cite[Solution to problem E1460]{carlitz1961solution} on congruence restrictions imposed on these polynomials. T\'oth obtained the product formula
\[
\phi_{\{e_1,e_k\}}(n) = n^k \prod_{p \mid n} \left(1 - \frac{1}{p}\right) \left(\left(1-\frac{1}{p}\right)^k - \frac{(-1)^k}{p^k}\right),
\]
where
\begin{equation}\label{def: toth e1ek}
\phi_{\{e_1,e_k\}}(n) \coloneqq \bigl|\{(a_1, \dots, a_k) \in \mathbb{Z}_n^k : \gcd(e_1(\mathbf{a}), n)=\gcd(e_k(\mathbf{a}),n) = 1 \}\bigr|.    
\end{equation}
Futhermore, Csizm\'azia and T\'oth~\cite{csizmazia2025generalizationseulersvarphifunctionrespect} introduced another generalization associated with a set of multivariable polynomials $F = \{f_1, \dots, f_m\}$ over $\mathbb{Z}_n^k$, defining
\begin{equation}\label{eqn:varphiF}
\varphi_F(n) \coloneqq \bigl|\{(a_1, \dots, a_k) \in \mathbb{Z}_n^k : \gcd(f_1(\mathbf{a}), \dots, f_m(\mathbf{a}), n) = 1 \}\bigr|,
\end{equation}
where $\mathbf{a} = (a_1, \dots, a_k)$. They proved that $\varphi_F$ is multiplicative with product formula~\cite[Theorem~2.1]{csizmazia2025generalizationseulersvarphifunctionrespect}
\begin{equation}\label{eqn:productformNFp}
\varphi_F(n)=n^k\prod_{p\mid n}\left(1-\frac{N_F(p)}{p^k}\right),
\end{equation}
where $N_F(p)$ denotes the number of common zeros of $F$ over $\mathbb{F}_p^k$. This extends earlier generalizations by  Stevens~\cite{stevens1971generalizations}.

A natural strengthening of equation (\ref{eqn:varphiF}) requires each $f_i(\mathbf{a})$ individually coprime to $n$:
\begin{equation}\label{eqn:phi}
\phi_F(n) \coloneqq \bigl|\{(a_1, \dots, a_k) \in \mathbb{Z}_n^k : \gcd(f_i(\mathbf{a}), n) = 1 \ \forall\, i=1,\dots,m \}\bigr|,
\end{equation}
which is a generalization of the totient function given in equation (\ref{def: toth e1ek}) studied by T\'{o}th.
Clearly, $\phi_F(n) \leq \varphi_F(n)$, and if $f = \prod f_i$ then $\phi_F(n) = \varphi_{\{f\}}(n)$, implying multiplicativity of $\phi_F$. Similar to the product formula (\ref{eqn:productformNFp}) one can easily prove 
\[
\phi_F(n)=n^k\prod_{p\mid n}\left(1-\frac{\widetilde{N}_F(p)}{p^k}\right),
\]
where $\widetilde{N}_F(p)$ counts the union of zeros of $F$ over $\mathbb{F}_p^k$. In Section~\ref{sect: phi_F}, we relate $N_F(p)$ and $\widetilde{N}_F(p)$---and hence $\varphi_F$ and $\phi_F$---via inclusion-exclusion (Theorem~\ref{thm: relation_tot}), yielding explicit product formulas for both.

In this paper, we consider the elementary symmetric polynomials defined as follows:
\[
e_j(x_1, \dots, x_k) = \sum_{1 \leq i_1 < \cdots < i_j \leq k} x_{i_1} \cdots x_{i_j},
\]
So, we use the alternate notation for $J=\{j_1,\dots, j_m\}\subseteq \{1,2,\dots,k\}$,
\begin{equation}\label{eq: varphi_J defn}
\varphi_J(n):=|\{(x_1,\dots, x_k): \gcd(e_{j_1}(x_1,\dots, x_k),\dots, e_{j_m}(x_1,\dots, x_k), n)=1\}|,    
\end{equation}
and,
\begin{equation}\label{eq: phi_J defn}
\phi_J(n)=\left|\left\{(x_1,\dots, x_k): \gcd\left(\prod_{j\in J}e_j(x_1,\dots, x_k), n\right)=1\right\}\right|.
\end{equation}

Our main contributions in this paper are as follows. We extend the analysis to the second elementary symmetric polynomial
\[
e_2(\mathbf{x}) = \sum_{1 \leq i < j \leq k} x_i x_j,
\]
deriving explicit product formulas for $\phi_F(n)$ and $\varphi_F(n)$ when $F$ includes $e_2$ and also prove related Menon-type identities for these totients. Our work further connects to the theory of restricted linear congruences---counting solutions to
\[
a_1 x_1 + \dots + a_k x_k \equiv b \pmod{n}
\]
subject to certain $\gcd$ conditions---as studied by Lehmer~\cite{lehmer1913certain}, Dixon~\cite{dixon1960finite}, Rearick~\cite{rearick1963linear}, and Bibak et al.~\cite{bibak2017restricted}, among others. In particular, T\'oth's $\phi_{\{e_1, e_k\}}(n)$ is always divisible by $\varphi(n)$, with the quotient counting solutions to $\sum x_i \equiv m \pmod{n}$ under the condition $\gcd(\prod x_i, n) = \gcd(m, n) = 1$. We establish analogous divisibility results for $e_2$-based generalized totients, thereby revealing deeper structural links between polynomial zeros over finite fields, restricted linear congruences, and generalized Euler totients.

\section{Results on the Function \texorpdfstring{$\varphi_J(\cdot)$}{the weaker totient function}}

 In this section, we derive explicit product forms for the functions $\varphi_{\{2\}}(n)$, $\varphi_{\{1,2\}}(n)$, $\varphi_{\{2,k\}}(n)$ and $\varphi_{\{1,2,k\}}(n)$ (See equation (\ref{eq: varphi_J defn})). For this purpose, we require a result on the number of solutions of an equation involving a quadratic form over a finite field $\mathbb{F}_p$,
A quadratic form $f$ in $k$ indeterminates over a finite field $\mathbb{F}_p$ is a homogeneous polynomial in $\mathbb{F}_p[x_1,\dots, x_k]$ of degree 2 or the zero polynomial. The quadratic form $f$ can be written as $f(x)=x^TAx$ for a symmetric matrix $A$. The quadratic form $f$ is non-degenerate if the associated symmetric matrix $A$ is non-singular.

The result for the product form for $\varphi_J(n)$ follows by identifying the appropriate matrix and determining the values of $k$ and $p$ for which it becomes degenerate.

\begin{thm}\label{thm: e2}
    \[
    \varphi_{\{2\}}(n)=n^k\prod_{p\mid n}\left(1-\frac{N_k(\{2\},p)}{p^k}\right)
    \]
    Where,
    for prime $p>2$ and $k>1$, the number of solutions $N_k(\{2\},p)$ of $e_2(a_1,\dots, a_k)\equiv 0\pmod{p}$ is:
    $$
    N_k(\{2\},p)=\begin{cases}
        p^{k-1}+(p-1)p^{(k-1)/2}\eta((-1)^{(k-1)/2}(1-\gcd(k-1,p))), &k\text{ odd},\\
        p^{k-1}+(p-1)p^{(k-2)/2}\eta((-1)^{k/2+1}(k-1)), &k\text{ even}.
    \end{cases}
    $$
    Also,
    $$
    N_k(\{2\},2)=\frac{1}{4}\left(2^{k+1}+2(\sqrt{2})^{k+1}\cos\left(\frac{\pi}{4}-\frac{k\pi}{4}\right)\right).
    $$
\end{thm}

\begin{proof}
For $p>2$, the quadratic form $e_2$ has matrix
$$
A= \begin{bmatrix}
0 &2^{-1} &\cdots &2^{-1}\\
2^{-1} & 0 & \cdots &2^{-1}\\
\vdots & \vdots & \ddots & \vdots\\
2^{-1} & 2^{-1} & \cdots & 0
\end{bmatrix}_{k\times k}.
$$

When $A$ is non-degenerate: $\det A=(-1)^{k-1}2^{-k}(k-1)\neq0$. By lemma~\ref{lem: zeros of quadratic forms},
$$
N_k(\{2\}, p)=\begin{cases}
p^{k-1}, &k\text{ odd}\\
p^{k-1}+(p-1)p^{(k-2)/2}\eta((-1)^{k/2+1}(k-1)), &k\text{ even}.
\end{cases}
$$

When $A$ is degenerate: $\det A=0$, so $k\equiv1\pmod{p}$. Nullspace is span$\{(1,\dots,1)^T\}$. Choose basis $\{e_2,\dots,e_k\}$ for complement, with $e_j$ having $-1$ at pos.~1, $1$ at pos.~$j$. Then $\overline{A}=(e_i^TAe_j)$ is $(k-1)\times(k-1)$ non-degenerate: $e_i^TAe_j=1$ ($i=j$), $2^{-1}$ ($i\neq j$), $\det\overline{A}=k/2^{k-1}\neq0$. By lemma~\ref{lem: zeros of quadratic forms},
$$
N_k(\{2\},p)=\begin{cases}
p^{k-1}+(p-1)p^{(k-1)/2}\eta((-1)^{(k-1)/2}), &k\text{ odd},\\
p^{k-1}, &k\text{ even}.
\end{cases}
$$

Combining cases yields the formula (adjusting Legendre symbol via $\gcd(k-1,p)$).

For $p=2$, let $v\in\{0,1\}^k$ have $j$ ones: $v^TA v=j(j-1)/2\equiv0\pmod{2}$ iff $j\equiv0,1\pmod{4}$. Thus,
$$
N_k(\{2\},2)=\sum_{\substack{j=0\\j\equiv0,1\pmod{4}}}^k\binom{k}{j}.
$$
Using roots of unity filter with $\omega_4=i$,
$$
\sum_{j\equiv0\pmod{4}}\binom{k}{j}=\frac{1}{4}\bigl(2^k+(1+i)^k+(1-i)^k\bigr)=\frac{1}{4}\bigl(2^k+2(\sqrt{2})^k\cos(k\pi/4)\bigr),
$$
$$
\sum_{j\equiv1\pmod{4}}\binom{k}{j}=\frac{1}{4}\bigl(2^k+2(\sqrt{2})^k\sin(k\pi/4)\bigr).
$$
Summing gives the formula.

Now, The product form is readily obtained by using equation (\ref{eqn:productformNFp}) proved in \cite[Theorem 2.1]{csizmazia2025generalizationseulersvarphifunctionrespect}.
\end{proof}

\begin{thm}\label{thm: e1,e2}
     \[
    \varphi_{\{1,2\}}(n)=n^k\prod_{p\mid n}\left(1-\frac{N_k(\{1,2\},p)}{p^k}\right)
    \]
    Where, for prime $p>2$ and $k>1$, the number of solutions $N_k(\{1,2\},p)$ of
    \begin{align*}
        e_j(a_1,\dots, a_k)&\equiv 0\pmod{p}, \hspace{1cm}j\in \{1,2\}
    \end{align*}
    is:
    $$
    N_k(\{1,2\},p)=\begin{cases}
       p^{k-2}+(p-1)p^{(k-3)/2}\eta\bigl((-1)^{(k-1)/2}k\bigr), &k\text{ odd},\\
       p^{k-2}+(p-1)p^{(k-2)/2}\eta\bigl((-1)^{k/2}(1-\gcd(k,p))\bigr), &k\text{ even}.
    \end{cases}
    $$
    Also,
    $$
    N_k(\{1,2\},2)=\frac{1}{4}\left(2^k+2(\sqrt{2})^k\cos\frac{k\pi}{4}\right).
    $$
\end{thm}

\begin{proof}
Under $e_1\equiv0\pmod{p}$, substitute $a_k\equiv -e_1(a_1,\dots,a_{k-1})\pmod{p}$ into $e_2\equiv0$ to obtain the quadratic form $
       a_1^2+\dots+a_{k-1}^2+e_2(a_1,\dots, a_{k-1})\equiv 0\text{ mod }p,$ in $(a_1,\dots,a_{k-1})$, with associated matrix $A$ satisfying $\det A=2^{1-k}k$.

When A is non-degenerate ($\det A\neq0$): lemma~\ref{lem: zeros of quadratic forms} gives
\[
N_k(\{1,2\},p)=\begin{cases}
p^{k-2}+(p-1)p^{(k-3)/2}\eta\bigl((-1)^{(k-1)/2}k\bigr), &k\text{ odd},\\
p^{k-2}, &k\text{ even}.
\end{cases}
\]

When A is degenerate ($\det A=0$, i.e., $k\equiv0\pmod{p}$): As in proof of theorem~\ref{thm: e2}, restrict to $(k-2)$-dimensional non-degenerate complement ($\det\overline{A}=(-2)^{2-k}(k-1)\neq0$), yielding
\[
N_k(\{1,2\},p)=\begin{cases}
p^{k-2}, &k\text{ odd},\\
p^{k-2}+(p-1)p^{(k-2)/2}\eta\bigl((-1)^{k/2}\bigr), &k\text{ even}.
\end{cases}
\]

The stated formula follows by combining cases and adjusting the Legendre symbol via $\gcd(k,p)$. The $p=2$ case follows analogously to theorem ~\ref{thm: e2} via roots-of-unity filter.

Now, the form of the product is easily obtained using equation (\ref{eqn:productformNFp}) proved in \cite[Theorem 2.1]{csizmazia2025generalizationseulersvarphifunctionrespect}.
\end{proof}
We state the lemma that was used to prove theorems (\ref{thm: e2}) and (\ref{thm: e1,e2}).
\begin{lem}\cite[Theorems 6.26, 6.27]{lidl1997finite}\label{lem: zeros of quadratic forms}
     Let $f$ be a non-degenerate quadratic form over $\mathbb{F}_p$, $p$ odd, in $k$ indeterminates. Then for $b\in \mathbb{F}_p$ the number of solutions $N(b)$ of $f(x_1,\dots, x_k)=b$ in $\mathbb{F}_p^n$ is \[
     N(b)=\begin{cases}
         p^{k-1}+p^{(k-1)/2}\eta((-1)^{(k-1)/2}b\Delta), &\text{ if }k\text{ is }odd\\
         p^{k-1}+\nu(b)p^{(k-2)/2}\eta((-1)^{k/2}\Delta), &\text{ if }k\text{ is }even.
     \end{cases}
     \]
     where $\eta$ ia the quadratic character of $\mathbb{F}_p$ and $\Delta=\text{det}(f)$. The integer valued function $\nu$ on $\mathbb{F}_p$ is defined by $\nu(b)=-1$ for $b\in \mathbb{F}_p^*$ and $\nu(0)=p-1$. 
\end{lem}
\begin{thm}\label{thm: including k}
Let $J\subseteq \{1,2,\dots, k-1\}, J\neq \varnothing$ then if the number of solutions $N_k(J,p)$ of \[
e_j(a_1,\dots, a_k)\equiv 0\text{ mod }p,
\]
for $j\in J$, is known then the number of solutions of $N_k(J\cup\{k\},p)$ is given by
\begin{equation}\label{eqn: recurrence add k}
N_k(J\cup \{k\}, p)=\sum_{j=1}^k(-1)^{j+1}\binom{k}{j}N_{k-j}(J/\{k-j+1,\dots, k-1\},p)    
\end{equation}

where we define $N_0(\varnothing,p)=1$. Here $A/B$ denotes the set difference of $A$ and $B$.

\end{thm}
\begin{proof}
    Since $e_k(a_1,\dots,a_k)\equiv 0\text{ mod }p$ means that atleast one $a_j\equiv 0\text{ mod }p$ so we count the union of all such solutions using the inclusion-exclusion principle. 
\end{proof}
\begin{cor}\label{thm: e2, ek}
    For prime $p>2$ and $k>2$, the number of solutions, $N_k(\{2,k\},p)$ of the congruence:
    \begin{align*}
        &e_j(a_1,\dots, a_k)\equiv 0\text{ mod }p, \hspace{1cm}j\in \{2,k\},
    \end{align*} is:
    \[
    N_k(\{2,k\},p)=\sum_{j=1}^{k-2}(-1)^{j+1}\binom{k}{j}N_{k-j}(\{2\},p)+(-1)^{k}(kp-1)
    \]
    Also,\[
    N_k(\{2,k\},2)=\frac{1}{4}\left(2^{k+1}+2(\sqrt{2})^{k+1}\cos\left(\frac{\pi}{4}-\frac{k\pi}{4}\right)\right)-1.
    \]
\end{cor}
\begin{proof}
    This follows from theorem (\ref{thm: including k}).
    Notice however that when the number of variable equivalent to $0$ is $k-1$ we have $p$ solutions. Again, when the number of such variables is $k$ then there is just $1$ solution.
    The case $p=2$ follows by a similar roots of unity argument as in theorem (\ref{thm: e2}).
\end{proof}
\begin{cor}
    For prime $p>2$ and $k>2$, the number of solutions, $N_k(\{1,2,k\},p)$ of the congruence:
    \begin{align*}
        &e_j(a_1,\dots, a_k)\equiv 0\text{ mod }p, \hspace{1cm}j\in \{1,2,k\},
    \end{align*} is:
    \[
    N_k(\{1,2,k\},p)=\sum_{j=1}^{k-2}(-1)^{j+1}\binom{k}{j}N_{k-j}(\{1,2\},p)+(-1)^k(k-1)
    \]
    Also,\[
    N_k(\{1,2,k\},2)=\frac{1}{4}\left(2^k+2(\sqrt{2})^k\cos\frac{k\pi}{4}\right)-1.
    \]
\end{cor}
\begin{proof}
    The proof is similar to the one for theorem (\ref{thm: e2, ek}). Only when there are $k-1$ variables are equivalent to $0 \text{ mod }p$ then the number of solutions is 1.
\end{proof}

\section{Results on the Function \texorpdfstring{$\phi_J(\cdot)$}{the stronger totient function}}\label{sect: phi_F}
In the introduction, we had defined the function $\phi_F(\cdot)$ for a set of polynomials $F=\{f_1,\dots, f_m\}$ with integer coefficients.
First, we note a connection between the functions $\varphi_F(\cdot)$ and $\phi_F(\cdot)$.
\begin{thm}[Relation between totient functions]\label{thm: relation_tot}
\[
\phi_F(p^k)=\sum_{\substack{J \subseteq F \\ J \ne \varnothing}}(-1)^{|J|+1}\varphi_J(p^k), \quad \varphi_F(p^k)=\sum_{\substack{J \subseteq F \\ J \ne \varnothing}}(-1)^{|J|+1}\phi_J(p^k).
\]
\end{thm}

\begin{proof}
By \eqref{eqn:productformNFp}, $\varphi_F(p^k)=p^k(1-N_F(p)/p^k)$. Let $f=\prod_{f_i\in F}f_i$. Then $N_{\{f\}}(p)=\sum_{J\subseteq F\setminus\emptyset}(-1)^{|J|+1}N_J(p)$ by inclusion-exclusion. Thus,
$$
\varphi_{\{f\}}(p^k)=p^k\left(1-\frac{\sum_{J\subseteq F}(-1)^{|J|+1}N_J(p)}{p^k}\right)=\sum_{\substack{J\subseteq F\\ J \ne \varnothing}}(-1)^{|J|+1}\varphi_J(p^k).
$$
The reverse follows symmetrically.
\end{proof}

In what follows, we use the alternate notation $\phi_J(n)$ defined in equation (\ref{eq: phi_J defn}).

\begin{thm}
    When $1\in J\subseteq\{1, \dots, k\}$, the following Menon's identity holds:\[
    \sum_{ (a_1,\dots,a_k)\in S}f(\gcd(a_1+\dots+a_k-1,n))=\phi_J(n)\sum_{d\mid n}\frac{\mu*f(d)}{\varphi(d)}
    \]
    where, $S=\left\{(x_1,\dots, x_k): \gcd\left(\prod_{j\in J}e_j(x_1,\dots, x_k), n\right)=1\right\}$.
\end{thm}
\begin{proof}
    Note that,
    \begin{align*}
        \sum_{(a_1,\dots, a_k)\in S}f(\gcd(a_1+\dots+a_k-1, n))=\sum_{\substack{a=1\\\gcd(a,n)=1}}^nf(\gcd(a-1, n))\sum_{\substack{(a_1,\dots, a_k)\in S\\a_1+\dots+a_k=a\text{ mod }n}}1.
    \end{align*}
    It can be seen with the help of simple mapping argument, as given in theorem (\ref{thm: even}), that the inner sum is fixed for all $a$ such that $\gcd(a, n)=1$ which is nothing but:\[
    \sum_{\substack{(a_1,\dots, a_k)\in S\\a_1+\dots+a_k=a\text{ mod }n}}1=\frac{\phi_J(n)}{\varphi(n)}
    \]
    And the rest follows the a generalization of the Menon's identity given in \cite[Theorem 4.2]{toth2021proofs}.
    \end{proof}

    We would like to point out here that a product form the function $\phi_{\{1,k\}}(n)$ is known through the work of T\'{o}th \cite{toth2022another}. It is evident that the function $\phi_{\{k-1,k\}}$ should have the same form:
    \[
    \phi_{\{k-1,k\}}(n)=\phi_{\{1,k\}}(n)=n^k\prod_{p\mid n}\left(1-\frac{1}{p}\right)\left(\left(1-\frac{1}{p}\right)^k - \frac{(-1)^k}{p^k}\right).
    \].
    In fact, more generally it is evident that the following should hold:\[
    \phi_{\{i,k\}}(n)=\phi_{\{k-i,k\}}(n).
    \]
\begin{thm}\label{thm: tot_func2{1,2}}
    For $k\geq 2$, the following identity holds for odd $n$:\[
    \phi_{\{1,2\}}(n)=n^k\prod_{p\mid n}\left(1-\frac{1}{p}-\frac{p-1}{p^2}+(p-1)\frac{h_k(p)}{p^k}\right)
    \]
    where\[
    h_k(p)=\begin{cases}
        p^{(k-3)/2}\eta((-1)^{(k-1)/2}k)-p^{(k-1)/2}\eta((-1)^{(k)}(1-\gcd(k-1, n))), &\text{ if }k\text{ is odd}\\
        p^{(k-2)/2}\eta((-1)^{k/2}(1-\gcd(k,p)))-
         p^{(k-2)/2}\eta((-1)^{k/2+1}(k-1)), &\text{ if }k\text{ is even}.
    \end{cases}
    \]
    Here $\eta(\cdot)$ is the quadratic character$\text{ mod } p$.
     Also,
    \[
    \phi_{\{1,2\}}(2^l)=2^{lk}\left(\frac{1}{4}-\frac{1}{2\left(\sqrt{2}\right)^k}\sin\frac{k\pi}{4}\right)
    \]
\end{thm}
\begin{proof}
    This theorem is obtained using theorem (\ref{thm: relation_tot}) as:
    \[
    \phi_{\{1,2\}}(n)=\varphi_{\{1\}}(n)+\varphi_{\{2\}}(n)-\varphi_{\{1,2\}}(n).
    \]
    This gives us:
    \[
    \phi_{\{1,2\}}(n)=n^k\prod_{p\mid n}\left(1-\frac{1}{p}+\frac{N_k(\{1,2\},p)-N_k(\{2\},p)}{p^k}\right)
    \]
    We get the result upon substituting the respective values of $N_k(\{2\},p)$ and $N_k(\{1,2\},p)$.
\end{proof}
It is also evident that using the same argument, we can obtain a product form for $\phi_{\{1,2,k\}}(n)$. However, a general form for this function is quite cumbersome.

\section{Solutions of Certain Restricted Linear Congruences}
For $a_1,\dots, a_k, b, n\in \mathbb{Z}$, $n\geq 1$, a linear congruence in $k$-unknowns $x_1,\dots, x_k$ is of the form 
\begin{equation}\label{eqn: lin_cong}
    a_1x_1+\dots+a_kx_k\equiv b\text{ mod }n.
\end{equation}
By a solution of (\ref{eqn: lin_cong}) we mean an ordered $k$-tuple of integers modulo $n$, denoted by $(x_1,\dots, x_k)$ that satisfies (\ref{eqn: lin_cong}). Additionally, there might restrictions on the set of the solutions of (\ref{eqn: lin_cong}). For instance, we might want to solve the linear congruence subject to certain conditions like $\gcd(x_i,n)=t_i, (1\leq i\leq k)$ where $t_1,\dots, t_k$ are positive divisors of $n$. This problem has been dealt by Bibak in \cite{bibak2017restricted}.

In what follows, we will outline a method that solves the restricted linear congruence problem subject to the $\gcd$ conditions on certain functions of the unknowns.
\begin{thm}\label{thm: even}
Let $F=\{f_1,\dots,f_m\}$ be homogeneous polynomials. Let $g_k(b,n)$ be the number of solutions to $a_1x_1+\dots+a_kx_k\equiv b\pmod{n}$ with $\gcd(f_i(\mathbf{x}),n)=1$ $\forall i$. Then $g_k(b,n)=g_k(\gcd(b,n),n)$.
\end{thm}

\begin{proof}
Let $d=\gcd(b,n)$, $r=b/d$. The map $(a_1,\dots,a_k)\mapsto(r^{-1}a_1,\dots,r^{-1}a_k)$ bijects solutions of $\sum a_ix_i\equiv b\pmod{n}$ to solutions of $\sum(r^{-1}a_i)x_i\equiv d\pmod{n}$, preserving the gcd conditions since $f_i$ are homogeneous.
\end{proof}

\begin{thm}\label{thm:rest_lin_cong}
Let $F=\{f_1,\dots,f_m\}$ be homogeneous polynomials. For $\gcd(b,n)=1$, let $g_k(n)$ be the number of solutions to $a_1x_1+\dots+a_kx_k\equiv b\pmod{n}$ with $\gcd(f_i(\mathbf{x}),n)=1$ $\forall i$. Then
$$
g_k(n)=\frac{\phi_{F'}(n)}{\varphi(n)},\quad F'=F\cup\{a_1x_1+\dots+a_kx_k\}.
$$
\end{thm}

\begin{proof}
The map $(x_1,\dots,x_k)\mapsto(bx_1,\dots,bx_k)$ bijects solutions of $\sum a_ix_i\equiv1\pmod{n}$ to solutions of $\sum a_ix_i\equiv b\pmod{n}$, preserving gcd conditions (homogeneity). Thus $g_k(n)=\phi_{F'}(n)/\varphi(n)$.
\end{proof}
\begin{exmp}
    For positive integers $m,n$ such that $\gcd(m,n)=1$, consider the following restricted linear congruence problem:
    \begin{equation}
        x_1+\dots+x_k\equiv m\text{ mod }n,
    \end{equation}
    subject to the constraint $\gcd(e_2(x_1,\dots, x_k), n)=1$.
    We use theorem (\ref{thm:rest_lin_cong}) and theorem (\ref{thm: tot_func2{1,2}}) to obtain the number of solutions $g_k(n)$ as:\[
    g_k(n)=\frac{\phi_{\{1,2\}}(n)}{\varphi(n)}.
    \]
    which is:
    \[
    g_k(n)=n^{k-1}\prod\left(\frac{1}{p}-\frac{1}{p^2}+\frac{h_k(p)}{p^k}\right).
    \]
    for odd $n$ and,
    \[
    g_k(2^l)=2^{lk-l+1}\left(\frac{1}{4}-\frac{1}{2\left(\sqrt{2}\right)^k}\sin\frac{k\pi}{4}\right).
    \]
\end{exmp}

To conclude, we remark that explicit product forms are derived here only for $e_2$-based totient generalizations; higher $e_l$ remain open. For $p=2$, $N_k(\{l\},2)$ counts solutions to $e_l(x_1,\dots,x_k)\equiv0\pmod{2}$. With $j$ ones, this is $\binom{j}{l}\equiv0\pmod{2}$, or zero for $j<l$.

Lucas' theorem \cite{lucas1878theorie} states that $
    \binom{j}{l}= \prod_{i=0}^m\binom{j_i}{l_i}\text{ mod }p,
    $
    where $l = \sum_{i=0}^ml_ip^i$ and $j = \sum_{i=0}^mj_ip^i$. This gives $\binom{j}{l}\equiv\prod\binom{j_i}{l_i}\pmod{2}$, so $\binom{j}{l}\equiv0\pmod{2}$ iff $l \wedge j \neq l$ where $\wedge$ is the bitwise AND operator on the binary representations of $l$ and $j$. Thus,
$$
N_k(\{l\},2)=\sum_{\substack{j=0\\j\wedge l\neq l}}^k\binom{k}{j}.
$$

\begin{exmp}
For $l=3$, $j\not\equiv3\pmod{4}$. Using roots-of-unity filter on $f(x)=x^2(1+x)^k$,
$$
\sum_{j\equiv2\pmod{4}}\binom{k}{j}=\frac{1}{4}\bigl(2^k-2(\sqrt{2})^k\cos(k\pi/4)\bigr).
$$
Combining with theorem~\ref{thm: e2} yields
$$
N_k(\{3\},2)=\frac{1}{4}\bigl(3\cdot2^k+2(\sqrt{2})^k\sin(k\pi/4)\bigr).
$$
\end{exmp}

\section*{Acknowledgement} 
The second author is supported by the SERB-MATRICS project (MTR/2023/000705)
from the Department of Science and Technology, India, for this work.

\section*{Data Availability}
Data sharing is not applicable to this article, as no datasets were generated or analyzed during the current study.

\bibliographystyle{alpha}
\bibliography{References}
\end{document}